\documentclass[11pt,reqno,a4paper]{amsart}
\usepackage{amsfonts}
\usepackage{epsfig}
\usepackage{graphicx}
\usepackage{amsmath}
\usepackage{amssymb}
\usepackage{colordvi}
\usepackage{times,amsmath,epsfig,float,multicol,subfigure}

\usepackage{amsfonts, amssymb, amsmath, amscd, amsthm, txfonts, color, enumerate}

 \usepackage{graphicx}
\usepackage{hyperref}

 \allowdisplaybreaks

\numberwithin{equation}{section}
\numberwithin{figure}{section}
\theoremstyle{plain}
\newtheorem{theorem}{Theorem}[section]

\newtheorem{lemma}[theorem]{Lemma}

\theoremstyle{definition}
\newtheorem{remark}[theorem]{Remark}
\newtheorem{example}[theorem]{Example}
\newtheorem{definition}[theorem]{Definition}

\numberwithin{equation}{section}

\begin{document}

\title[Genericity of homeomorphisms with full mean Hausdorff dimension]{Genericity of homeomorphisms with full mean Hausdorff dimension}
\author{Jeovanny M. Acevedo}
\date{\today}

\address{Jeovanny de Jesus Muentes Acevedo, Facultad de Ciencias B\'asicas,  Universidad Tecnol\'ogica de  Bol\'ivar, Cartagena de Indias - Colombia}
\email{jmuentes@utb.edu.co}

\maketitle

\begin{abstract}   It is well known that the presence of horseshoes leads to positive entropy. If our goal is to construct a continuous map with infinite entropy, we can consider an infinite sequence of horseshoes, ensuring an unbounded number of legs.

Estimating the exact values of both the metric mean dimension and mean Hausdorff dimension for a homeomorphism is a challenging task. We need to establish a precise relationship between the sizes of the horseshoes and the number of appropriated legs to control both quantities.

Let $N$ be an $n$-dimensional compact Riemannian manifold, where $n \geq 2$, and $\alpha \in [0, n]$. In this paper, we construct a homeomorphism $\phi: N \rightarrow N$ with   mean Hausdorff dimension equal to $\alpha$. Furthermore, we prove that the set of homeomorphisms on $N$ with both lower and upper mean Hausdorff dimensions equal to $\alpha$ is dense in $\text{Hom}(N)$. Additionally, we establish that the set of homeomorphisms with   upper mean Hausdorff dimension equal to $n$ contains a residual subset of  $\text{Hom}(N).$ 
\end{abstract}

\section{Introduction}

 Throughout this work,   $N$ will be denote a compact topological space with a fixed metric $d$ and $ \phi: N\rightarrow N$ a continuous map.  
 An useful tool to classify dynamical systems is the topological entropy, introduced by  R. L. Adler, A. G. Konheim and M. H. McAndrew in 1965: if two dynamical systems are topologically conjugated, then they have the same topological entropy.    However, 
in (1980), Koichi Yano     proved  that if $N$ is a compact manifold with $\text{dim}(N)\geq 2$, then  the set $\mathcal{A}$ consisting of homeomorphisms $ \phi: N\rightarrow N$ with infinite topological entropy   is residual in $\text{Hom}(N) $ (see \cite{Yano}). Hence,   topological entropy fails to differentiate dynamical systems in $\mathcal{A}$. 

 \medskip
 
In 1999,  M. Gromov in \cite{Gromov}   introduced the notion of mean topological dimension for  continuous maps,   which  is an invariant under topological  conjugacy. If one dynamical system has positive mean topological dimension, then it has infinite topological entropy.    Therefore, mean topological dimension is an useful tool to classify dynamical systems with infinite topological entropy. For instances, the studies presented in \cite{lind}, \cite{elon}, \cite{gutman}, \cite{Tsukamoto2} and \cite{GTM} illustrate that minimal systems with a mean topological dimension less than $\frac{n}{2}$ can be effectively embedded into the left-shift map on $([0,1]^n)^\mathbb{Z}$.  Some applications in information theory can be found in  \cite{lind3} and \cite{Lindestr}. We will denote by $\text{mdim}(N,\phi)$ the mean topological  dimension of $\phi$.

\medskip 

Mean topological dimension is difficult to calculate. Therefore,  Lindenstrauss and Weiss   in \cite{lind} introduced the notion of   metric mean dimension. We will denote by  $\underline{\text{mdim}}_{\text{M}}(N,d,\phi)$   the lower metric mean dimension of  $\phi$ and  
 $ \overline{\text{mdim}}_{\text{M}}(N,d,\phi)$     the upper metric mean dimension of  $\phi$. Metric mean dimension  depends on the metric $d$, therefore it is not an invariant under topological conjugacy. Furthermore,  it is zero for any map with finite topological entropy. In \cite{lind}, Theorem 4.2, is proved that it  is an  upper bound for the mean topological  dimension, that is, we have that   
\begin{equation*}\label{12fsffv}
\text{mdim}(N,\phi)\leq     \underline{\text{mdim}}_{\text{M}}(N,d,\phi)\leq  \overline{\text{mdim}}_{\text{M}}(N,d,\phi) . \end{equation*}  
It is conjectured that, for any dynamical system $\phi:N\rightarrow N$,  there exists a metric $d$ on $N$ such that $\text{mdim}(N,\phi)=  \overline{\text{mdim}}_{\text{M}}(N,d,\phi)$. 

\medskip

The notion of metric mean dimension has been extensively studied, as we can see in the works  \cite{Backes}, \cite{Cheng}, \cite{Ma}, \cite{Yang}, among other works.  Recently,  Lindenstrauss and Tsukamoto, in \cite{LT2019},  defined a new tool which provides a better upper bound for the mean topological dimension: the mean Hausdorff dimension. Hence, it is very likely that this will become the subject of further studies. Mean Hausdorff dimension is not a  topological invariant, because it also depends on the metric $d$. We will denote by $\underline{\text{mdim}}_{\text{H}}(N,d,\phi) $ and $\overline{\text{mdim}}_{\text{H}}(N,d,\phi)$ the upper and lower mean Hausdorff dimension of $\phi$, respectively. 

\medskip

In   \cite{LT2019}, Proposition 3.2 (see also \cite{Tsukamoto}, Proposition 2.1),  was proved that     
\begin{equation}\label{mnqw}
 \text{mdim}(N,\phi)\leq  \overline{\text{mdim}}_{\text{H}}(N,d,\phi)  \leq  \underline{\text{mdim}}_{\text{M}}(N,d,\phi) . \end{equation}  
For mean topological dimension and metric mean dimension we have, respectively, that $$   {\text{mdim}} (N^{\mathbb{Z}} ,\sigma)\leq \text{dim}(N) \quad\text{and}\quad  \underline{\text{mdim}}_{\text{M}}(N^{\mathbb{Z}} ,\text{\textbf{d}},\sigma)=\text{dim}_{\text{B}}(N,d), $$ where $\text{dim}(N)$ is the topological dimension of $N$ and  $\text{dim}_{\text{B}}(N,d)$ is the box dimension of $N$,  $\sigma:N^{\mathbb{Z}} \rightarrow N^{\mathbb{Z}}  $ is the left shift map and $ {\textbf{d}}(\bar{x},\bar{y}) = \underset{j \in\mathbb{Z}}{\sum} \frac{1}{2^j}d(x_k,y_k)$ for all $\bar{x} = (x_k),\bar{y}=(y_k) \in N^{\mathbb{Z}}$   (see \cite{lind}, Theorem 3.1, and \cite{VV}, Theorem 5). In \cite{Tsukamoto2}, Theorem 1.1 presents the exact value of ${\text{mdim}}(N^{\mathbb{Z}}, \sigma)$ depending on the type of space $N$.     Recently, in \cite{MuentesAMJ}, Theorem 3.5, is proved that  $$ \text{dim}_{\text{H}}(N,d)\leq \underline{\text{mdim}}_{\text{H}}(N^{\mathbb{Z}} ,\text{\textbf{d}},\sigma) ,$$ where  $ \text{dim}_{\text{H}}(N,d)$ is the Hausdorff dimension of $N$.   Several interesting results, analog to  Furstenberg’s theorem (see \cite{Furstenberg},   Proposition III.1), which  relate the metric mean dimension and mean Hausdorff dimension of sub-shifts with  the topological entropy of $\mathbb{N}^{2}$-actions,   are presented in  \cite{Tsukamoto}.

\medskip

In \cite{Liu}, the concepts of mean packing dimension and mean pseudo-packing dimension for dynamical systems are introduced. The authors provide some basic properties of these notions and  formulas applicable to Cartesian products of dynamical systems. However, the mean Hausdorff dimension remains a more accurate approximation of the mean topological dimension (the mean Hausdorff dimension of a dynamical system is lower than its mean packing dimension and its mean pseudo-packing dimension).

\medskip

Obtaining a homeomorphism with mean Hausdorff dimension or  metric mean dimension     equal to a fixed $\alpha $ poses a significant challenge and is not a straightforward task. It requires establishing a precise relationship between the sizes of horseshoes and the corresponding number of appropriated legs to effectively control both quantities. Suppose that $N$ is an $n$-dimensional compact Riemannian manifold and fix $\alpha\in[0,n]$. In \cite{MuentesA}, we construct homeomorphisms defined on $N$ with metric mean dimension equal to $\alpha$. 
In this paper we construct homeomorphisms defined on $N$ with   mean Hausdorff dimension equal to $\alpha$. 
Furthermore, we will prove that:
\begin{itemize}
  \item  The set consisting of continuous maps on   $[0,1]$ with upper and lower mean Hausdorff dimension equal to $\alpha\in[0,1]$ is dense in  $C^{0}([0,1])$ (see Theorem \ref{densitypositivemanifoldj}).
\item  The set consisting of continuous maps on the interval $[0,1]$ with upper   mean Hausdorff dimension equal to $1$  is residual in  $C^{0}([0,1])$ (see Theorem \ref{teoresjjjjidual}).
 \item The set consisting of homeomorphisms on   $N$, with upper and lower  mean Hausdorff dimension equal to $\alpha\in[0,n]$ is dense in  $\text{Hom}(N)$ (see Theorem \ref{densitypositivemanifold}).
\item  The set consisting of homeomorphisms on $N$ with upper   mean Hausdorff dimension equal to $n$  contains a residual subset of   $\text{Hom}(N)$ (see Theorem \ref{teoeeeresjjjjidual}).
\end{itemize}

\section{Continuous maps on  the interval with positive  mean Hausdorff dimension} 
In this section we will construct continuous maps on the interval with   mean Hausdorff dimension equal to a fixed value $\alpha\in[0,1]$. 
We will recall the definition of mean Hausdorff dimension.  For a continuous map $\phi:N \to N$ and any  integer
$n\geq 1$ we define  the metric 
\begin{equation}\label{jdkf}
d_n(x,y)=\max \left\{d(x,y),d(\phi(x),\phi(y)),\dots,d(\phi^{n-1}(x),\phi^{n-1}(y))\right\},\quad\text{for }x,y\in N.
\end{equation}  
 For any $n\in \mathbb{N}$ and $A\subset N$, set $\text{diam}_{d_{n}}(A)$ the diameter of $A$ with the metric $d_{n}$. For $s\geq 0$ and $\varepsilon >0$, set
$$ \text{H}_{\varepsilon}^{s} (N,d_{n},\phi)=\inf\left\{ \sum_{i=1}^{\infty}(\text{diam}_{d_{n}} E_{i})^{s}: N=\bigcup_{i=1}^{\infty} E_{i} \text{ with } \text{diam}_{d_{n}} E_{i}\leq \varepsilon\text{ for all }i\geq 1\right\}. $$ By  convention we consider $0^{0}=1$ and $\text{diam}_{d_{n}}(\emptyset)^{s}=0$. When $N$ is compact, we have $$ \text{H}_{\varepsilon}^{s} (N,d_{n},\phi)=\inf\left\{ \sum_{i=1}^{k}(\text{diam}_{d_{n}} E_{i})^{s}: N=\bigcup_{i=1}^{k} E_{i} \text{ with } \text{diam}_{d_{n}} E_{i}\leq\varepsilon\text{ for all }i=1,\dots k\right\} $$
(see \cite{Falconer}, Section 2.4). Take $$ \text{dim}_{\text{H}}(N,d_{n},\phi,\varepsilon)=\sup\{s \geq 0:  \text{H}_{\varepsilon}^{s} (N,d_{n},\phi) \geq 1\}.  $$

\begin{definition}
The \textit{upper mean Hausdorff dimension} and \textit{lower mean Hausdorff dimension} of $(N,d,\phi) $ are defined respectively as
$$ \overline{\text{mdim}}_{\text{H}}(N,d,\phi)=\lim_{\varepsilon\rightarrow 0} \left(\limsup_{n\rightarrow \infty}\frac{1}{n} \text{dim}_{\text{H}}(N,d_{n},\phi,\varepsilon)\right), $$ $$\underline{\text{mdim}}_{\text{H}}(N,d,\phi)=\lim_{\varepsilon\rightarrow 0} \left(\liminf_{n\rightarrow \infty}\frac{1}{n} \text{dim}_{\text{H}}(N,d_{n},\phi,\varepsilon)\right). $$
\end{definition}

   For a fixed $\varepsilon >0$,  a subset $A\subset N$ is   $(n,\phi,\varepsilon)$-\textit{separated}
if $d_n(x,y)>\varepsilon$, for any two  distinct points  $x,y\in A$. Let $\text{sep}(n,\phi,\varepsilon)$ be the maximal cardinality of any $(n,\phi,\varepsilon)$-separated
subset of $N$ and set $$\text{sep}(\phi,\varepsilon)=\underset{n\to\infty}\limsup \frac{1}{n}\log \text{sep}(n,\phi,\varepsilon).$$  The \emph{lower  metric mean dimension}    and the \emph{upper metric mean dimension} of $(N,d,\phi)$ are defined, respectively, by
  \begin{equation*}\label{metric-mean}
 \underline{\text{mdim}}_{\text{M}}(N,d,\phi)=\liminf_{\varepsilon\to0} \frac{\text{sep}(\phi,\varepsilon)}{|\log \varepsilon|} \quad\text{and}\quad 
\overline{\text{mdim}}_{\text{M}}(N,d,\phi)=\limsup_{\varepsilon\to0} \frac{\text{sep}(\phi,\varepsilon)}{|\log \varepsilon|}.
\end{equation*}

\begin{remark}
Throughout this work, we will write $\text{mdim}_{\text{M}} (N, d, \phi)$ to refer to both notions \newline $\underline{\text{mdim}}_{\text{M}} (N, d,\phi)$ or $\overline{\text{mdim}}_{\text{M}} (N, d,\phi)$ and, we will write $\text{mdim}_{\text{H}} (N, d,\phi)$ to refer to both notions $\underline{\text{mdim}}_{\text{H}} (N, d,\phi)$ or  $\overline{\text{mdim}}_{\text{H}} (N, d,\phi)$.
\end{remark}

Next, we present some basic properties of the mean Hausdorff dimension. 
\begin{itemize}
    \item For any $p\in\mathbb{N}$ we have
 that (see \cite{MuentesAMJ}, Proposition 3.1) $$  {\text{mdim}_{\text{H}}}(N,d,\phi^{p})\leq p\,  {\text{mdim}_{\text{H}}}(N,d,\phi).
  $$ 
  \item  For any  two continuous maps $\phi:(N,d)\rightarrow (N,d)$ and $\psi:(M,d^{\prime})\rightarrow (M,d^{\prime})$ we have that (see \cite{MuentesAMJ}, Proposition 3.4)
$$\underline{\text{mdim}}_{\text{H}}(N \times M,  d \times d^{\prime}, \phi \times \psi) \geq \underline{\text{mdim}}_{\text{H}}(N, d, \phi) + \underline{\text{mdim}}_{\text{H}}(M, d^{\prime} , \psi).$$   \item If  $\mathbb{K}$ is $\mathbb{Z}$ or $\mathbb{N}$, $\sigma:N^{\mathbb{K}} \rightarrow N^{\mathbb{K}}  $ is the left shift map and $ {\textbf{d}}(\bar{x},\bar{y}) = \underset{j \in\mathbb{K}}{\sum} \frac{1}{2^j}d(x_k,y_k)$ for all $\bar{x} = (x_k),\bar{y}=(y_k) \in N^{\mathbb{K}}$, we have that (see \cite{MuentesAMJ}, Theorem 3.5) $$ \text{dim}_{\text{H}}(N,d)\leq \underline{\text{mdim}}_{\text{H}}(N^{\mathbb{K}} ,\text{\textbf{d}},\sigma) .$$    
\end{itemize}

An $s$-\textit{horseshoe}   for a continuous map $\phi:[0,1]\rightarrow [0,1]$ is an interval $J\subseteq [0,1]$ which has a partition into $s$ subintervals $J_{1},\dots,J_{s}$ such that $J_{j}^{\circ}\cap  J_{i}^{\circ}=\emptyset$ for $i\neq j$ and $J\subseteq \phi (\overline{J}_{i})$ for each $i=1,\dots, s$. In       \cite{VV}, Lemma 6, is proved   if $I_{k}=[a_{k-1},a_{k}]\subseteq [0,1]$ is an  $s_{k}$-horseshoe for a continuous map $\phi:[0,1]\rightarrow [0,1]$ consisting of $s_{k}$ subintervals with the same length $I_{k}^{1},\dots, I_{k}^{s_{k}}$, then  
$$ \text{sep}(  \phi  , \varepsilon_{k})\geq \log(s_{k}/2)\quad\text{where }\varepsilon_{k}=\frac{|I_{k}|}{s_{k}}. $$ 
It follows from this fact that  \begin{equation}\label{eefeff}\overline{\text{mdim}}_{\text{M}}([0,1],|\cdot|,\phi)\geq   \underset{k\rightarrow \infty}{\limsup}\frac{ 1}{\left|1-\frac{\log |I_{k}|}{\log s_{k}}\right|} , \end{equation} where $|\cdot|$ is the metric induced from the absolute value 
(see \cite{VV}, Proposition 8).   Next, we will present some analogous results for mean Hausdorff dimension.

\begin{lemma}\label{mvefe} Suppose  for each $k\in\mathbb{N}$ there exists    a   $s_{k}$-horseshoe for $\phi\in C^{0}([0,1])$,  $I_{k}=[a_{k-1},a_{k}]\subseteq [0,1]$, consisting of sub-intervals    $I_{k}^{1}, I_{k}^{2},\dots, I_{k}^{s_{k}} $ with the same length, where $s_{k}\geq 2$ for all $k\geq 1$. We have  \begin{equation}\label{eqas1} \overline{\emph{mdim}}_{\emph{H}}([0,1] ,|\cdot |,\phi) \geq  \underset{k\rightarrow \infty}{\limsup}\frac{ 1}{\left|1-\frac{\log |I_{k}|}{\log s_{k}}\right|} . \end{equation} 
\end{lemma} 
 \begin{proof} Note if $\varphi, \psi:[0,1]\rightarrow [0,1]$ are continuous maps such that $I_{k}\subseteq\varphi(I_{k}^{t})$ and $I_{k}=\psi(I_{k}^{t})$ for all $k\in\mathbb{N}$ and $t=1,\dots, s_{k}$, then  \begin{equation*}\label{eqajs1} \overline{\text{mdim}}_{\text{H}}([0,1] ,|\cdot |,\varphi) \geq  \overline{\text{mdim}}_{\text{H}}([0,1] ,|\cdot |,\psi) . \end{equation*} Hence, without loss of generality, we can assume that $\phi(I_{k}^{t})=I_{k}$, for each $k\in\mathbb{N}$ and $i=1,\dots, s_{k}$. 
We will fix  $\varepsilon >0$ and $s\geq 0$. For any $k,n\in \mathbb{N}$, we have 
$$ \text{H}_{\varepsilon}^{s} ([0,1],d_{n},\phi)\geq \text{H}_{\varepsilon}^{s} (I_{k},d_{n},\phi|_{I_{k}}). $$ Furthermore, if $\delta <\varepsilon, $ we have $\text{H}_{\varepsilon}^{s} (I_{k},d_{n},\phi|_{I_{k}})\leq \text{H}_{\delta}^{s} (I_{k},d_{n},\phi|_{I_{k}})$. For any $k\in\mathbb{N}$, set $$\varepsilon_{k}=\frac{|I_{k}|}{s_{k}}=|I_{k}^{t}|\quad\text{for }t=1,2,\dots, s_{k}.$$  Let $k\in\mathbb{N}$ such that $\varepsilon_{k+1} \leq \varepsilon\leq \varepsilon_{k}$. Hence, 
$$\text{H}_{\varepsilon_{k+1}}^{s} (I_{k},d_{n},\phi|_{I_{k}})\geq \text{H}_{\varepsilon}^{s} (I_{k},d_{n},\phi|_{I_{k}})\geq \text{H}_{\varepsilon_{k}}^{s} (I_{k},d_{n},\phi|_{I_{k}}). $$
We can divide each $I_{k}^{t}$ into $s_{k}^{n}$ subintervals, $\{J_{1}^{t},\dots, J_{s_{k}^{n}}^{t}\}$, such that, for each $l=1,\dots, s_{k}$, there exist   $m_{1},\dots,m_{s_{k}^{n-1}}\in \{1,\dots, s_{k}^{n}\}$ with  $$\phi^{n}(J_{m_{h}}^{t})=I_{k}^{l},\quad\text{for each }t=1,\dots, s_{k}, h=1,\dots, s_{k}^{n-1}.$$ Hence,  for each $m\in \{1,\dots, s_{k}^{n}\}$, we have 
$$ \text{diam}_{d_{n}}(J_{m}^{t})=\varepsilon_{k}\quad\text{ and therefore }\quad\text{H}_{\varepsilon_{k}}^{s} (I_{k},d_{n},\phi|_{I_{k}})\leq \sum_{t=1}^{s_{k}}\sum_{i=1}^{s_{k}^{n}}( \text{diam}_{d_{n}}(J_{i}^{t}))^{s}=s_{k}^{n+1}(\varepsilon_{k})^{s}.$$
Now, give that $\phi(I_{k}^{t})=I_{k}$, we have $|\phi(I_{k}^{t})|=s_{k}|I_{k}^{t}|=s_{k}\varepsilon_{k}.$ Hence, for each $n\in\mathbb{N}$, we need at least $s_{k}^{n}$ subintervals with diameter less that $\varepsilon_{k}$ to cover each $I_{k}^{t}$ and, therefore, $s_{k}^{n+1}$ subintervals with diameter less that $\varepsilon_{k}$ to cover each $I_{k}$. Since every $J_{m}^{t}$ has the same $d_{n}$-diameter, we have  \begin{equation}\label{nvdvevem2q}\text{H}_{\varepsilon_{k}}^{s} (I_{k},d_{n},\phi|_{I_{k}})=s_{k}^{n+1}(\varepsilon_{k})^{s}.\end{equation}
Next, if $s\leq -\frac{(n+1)\log (s_{k})}{\log(\varepsilon_{k})}$, then $s_{k}^{n+1}(\varepsilon_{k})^{s}\geq 1$. Thus, 
$$  \text{dim}_{\text{H}}(I_{k},d_{n},\phi,\varepsilon_{k})=\frac{(n+1)\log (s_{k})}{\log(\frac{1}{\varepsilon_{k}})}.$$
Consequently, 
\begin{align*}\overline{\text{mdim}}_{\text{H}}([0,1] ,|\cdot |,\phi) &= \lim_{\varepsilon\rightarrow 0} \left(\limsup_{n\rightarrow \infty}\frac{1}{n} \text{dim}_{\text{H}}([0,1],d_{n},\phi,\varepsilon)\right)\\
&\geq \limsup_{k\rightarrow \infty} \left(\limsup_{n\rightarrow \infty}\frac{1}{n} \text{dim}_{\text{H}}(I_{k},d_{n},\phi,\varepsilon_{k})\right)\\
&\geq  \underset{k\rightarrow \infty}{\limsup}\left(\limsup_{n\rightarrow \infty}\frac{1}{n}\frac{(n+1)\log (s_{k})}{\log(\frac{1}{\varepsilon_{k}})}  \right)=\underset{k\rightarrow \infty}{\limsup}\left( \frac{\log (s_{k})}{\log(\frac{s_{k}}{|I_{k}|})}  \right)\\
&=  \underset{k\rightarrow \infty}{\limsup}\frac{ 1}{\left|1-\frac{\log |I_{k}|}{\log s_{k}}\right|},\end{align*}
 which proves the lemma.
\end{proof}

    Next lemma provides examples of continuous maps on the interval with mean Hausdorff dimension equal to fixed  value in $[0,1]$.

 \begin{lemma}\label{misiu}
Suppose  for each $k\in\mathbb{N}$ there exists    a   $s_{k}$-horseshoe for $\phi\in C^{0}([0,1])$,  $I_{k}=[a_{k-1},a_{k}]\subseteq [0,1]$, consisting of sub-intervals   with the same length  $I_{k}^{1}, I_{k}^{2},\dots, I_{k}^{s_{k}} $ and $[0,1]=\cup_{k=1}^{\infty}I_{k}$. 
We can rearrange the intervals and suppose that  $2\leq s_{k}\leq s_{k+1}$ for each $k$. If each $\phi|_{I_{k}^{i}}:I_{k}^{i}\rightarrow I_{k} $ is a bijective  affine  map for all $k$ and $i=1,\dots, s_{k}$, we have \begin{enumerate}[i.] \item  $   {\underline{\emph{mdim}}_{\emph{H}}}([0,1] ,|\cdot |,\phi) =  \underset{k\rightarrow \infty}{\liminf}\frac{1}{\left|1-\frac{\log |I_{k}|}{\log s_{k}}\right|}. $  
\item If the limit $\underset{k\rightarrow \infty}{\lim}\frac{1}{\left|1-\frac{\log |I_{k}|}{\log s_{k}}\right|}$ exists, then $ {\emph{mdim}}_{\emph{H}}([0,1] ,|\cdot |,\phi) =  \underset{k\rightarrow \infty}{\lim}\frac{1}{\left|1-\frac{\log |I_{k}|}{\log s_{k}}\right|}.$
 \end{enumerate}
 \end{lemma}
 \begin{proof}   In \cite{Muentes}, Theorem 3.3, is proved that  $   {\underline{\text{mdim}}_{\text{M}}}([0,1] ,|\cdot |,\phi) \leq  \underset{k\rightarrow \infty}{\liminf}\frac{1}{\left|1-\frac{\log |I_{k}|}{\log s_{k}}\right|}. $  Therefore,   $$   {\underline{\text{mdim}}_{\text{H}}}([0,1] ,|\cdot |,\phi) \leq   {\underline{\text{mdim}}_{\text{M}}}([0,1] ,|\cdot |,\phi) \leq  \underset{k\rightarrow \infty}{\liminf}\frac{1}{\left|1-\frac{\log |I_{k}|}{\log s_{k}}\right|}.$$ 
Next, we will prove ${\underline{\text{mdim}}_{\text{H}}}([0,1] ,|\cdot |,\phi)\geq  \underset{k\rightarrow \infty}{\liminf}\frac{1}{\left|1-\frac{\log |I_{k}|}{\log s_{k}}\right|}$. We will fix  $\varepsilon >0$ and $s\geq 0$.   For any $k\in\mathbb{N}$, set $$\varepsilon_{k}=\frac{|I_{k}|}{s_{k}}=|I_{k}^{t}|\quad\text{for }t=1,2,\dots, s_{k}.$$  Let $k\in\mathbb{N}$ such that $\varepsilon_{k+1} \leq \varepsilon\leq \varepsilon_{k}$. We have (see \eqref{nvdvevem2q})
$$  \text{H}_{\varepsilon}^{s} ([0,1],d_{n},\phi )\geq \text{H}_{\varepsilon_{k}}^{s} (I_{k},d_{n},\phi|_{I_{k}})  =s_{k}^{n+1}(\varepsilon_{k})^{s}.$$
Next, if $s\leq -\frac{(n+1)\log (s_{k})}{\log(\varepsilon_{k})}$, then $s_{k}^{n+1}(\varepsilon_{k})^{s}\geq 1$. Thus, 
$$  {\text{dim}_{\text{H}}([0,1],d_{n},\phi,\varepsilon)}\geq \text{dim}_{\text{H}}(I_{k},d_{n},\phi,\varepsilon_{k})=\frac{(n+1)\log (s_{k})}{\log(\frac{1}{\varepsilon_{k}})}.$$
Consequently, 
\begin{align*}\underline{\text{mdim}}_{\text{H}}([0,1] ,|\cdot |,\phi) &= \lim_{\varepsilon\rightarrow 0} \left(\liminf_{n\rightarrow \infty}\frac{1}{n} \text{dim}_{\text{H}}([0,1],d_{n},\phi,\varepsilon)\right)\\
&\geq \liminf_{k\rightarrow \infty} \left(\liminf_{n\rightarrow \infty}\frac{1}{n} \text{dim}_{\text{H}}(I_{k},d_{n},\phi,\varepsilon_{k})\right)\\
&=  \underset{k\rightarrow \infty}{\liminf}\frac{ 1}{\left|1-\frac{\log |I_{k}|}{\log s_{k}}\right|},\end{align*}
 which proves i.

ii.   In \cite{Muentes}, Theorem 3.3, is proved if the limit $\underset{k\rightarrow \infty}{\lim}\frac{1}{\left|1-\frac{\log |I_{k}|}{\log s_{k}}\right|}$ exists, then $$\overline{\text{mdim}}_{\text{M}}([0,1] ,|\cdot |,\phi) =  \underset{k\rightarrow \infty}{\lim}\frac{1}{\left|1-\frac{\log |I_{k}|}{\log s_{k}}\right|}.$$  Therefore, ii. follows from this fact and i.
 \end{proof}

 Next examples follows from the above facts (see \cite{Muentes}, Examples 3.1, 3.5 and 3.6).

  \begin{example}\label{EXAMPLE1} Fix  $r\in(0,\infty)$. Set $a_{0}=0$ and $a_{n}= \sum_{i=0}^{n-1}\frac{C}{3^{ir}}$ for $n\geq 1$, where $C=\frac{1}{\sum_{i=0}^{\infty}\frac{1}{3^{ir}}}= \frac{3^{r}-1}{3^{r}}$. For each $n\geq 0$, let 
 $T_{n}: I_{n}:=[a_{n},a_{n+1}] \rightarrow [0,1] $ be the unique increasing affine map from $I_{n}$   onto $[0,1]$.  
For $s\in \mathbb{N}$, set $\phi_{s,r}:[0,1]\rightarrow [0,1]$, given by   $\phi_{s,r}|_{I_{n}}= T_{n}^{-1}\circ g^{s(n+1)}\circ T_{n}$ for any $n\geq 0$,  where $g:[0,1]\rightarrow [0,1]$, is defined by $x\mapsto |1-|3x-1||$.  It follows from Lemmas \ref{mvefe} and \ref{misiu} that $$ {{\text{mdim}}_{\text{H}}}([0,1] ,|\cdot |,\phi_{s,r}) =\frac{s}{r+s}\quad\text{for any }s\in \mathbb{N}. $$ 
  \end{example}

  \begin{example}\label{med1}   Set $a_{0}=0$ and $a_{n}= \sum_{i=1}^{n}\frac{6}{\pi^{2}i^{2}}$ for $n\geq 1$. Set  $I_{n}:=[a_{n-1},a_{n}]$ for any $n\geq 1$.    Let $\varphi\in C^{0}([0,1])$ be defined by  $\varphi |_{I_{n}} = T_{n}^{-1}\circ g^{n}\circ T_{n}$ for any $n\geq 1$, where $T_{n}$ and $g$ are as in  Example \ref{EXAMPLE1}.   For   $\varphi^{s}$, with $s\in\mathbb{N},$  we have   $s_{k}=3^{sk}$ for each $k\in\mathbb{N}$. We have $$ {{\text{mdim}}_{\text{H}}}([0,1] ,|\cdot |,\varphi^{s})  =1\quad\text{for any }s\in \mathbb{N}. $$
  \end{example}

   \begin{example}\label{EXAMPLE134}    Take $I_{n}=[a_{n-1},a_{n}]$ as in the above example.   Divide each interval $I_{n}$ into $2n+1$ sub-intervals with the same lenght, $I_{n}^{1}$, $\dots$, $I_{n}^{2n+1}$. For $k=1,3,\dots, 2n+1$, let $\psi|_{I_{n}^{k}} :I_{n} ^{k}\rightarrow I_{n} $ be  the unique increasing affine map from $I_{n} ^{k}$ onto $I_{n}$ and for $k=2,4,\dots, 2{n}$, let $\psi|_{I_{n}^{k}} :I_{n} ^{k}\rightarrow I_{n} $ be  the unique decreasing affine map from $I_{n} ^{k}$ onto $I_{n}$. We have   $$  {\text{mdim}}_{\text{H}}([0,1] ,|\cdot |,\psi^{s})  =\frac{s}{s+2}.$$  
  \end{example}

\section{Density of continuous maps on the interval with positive mean Hausdorff  dimension}\label{section6} 
 
 On $C^{0}([0,1])$ we will consider the metric \begin{equation*}\label{cbenfn} \hat{d}(\phi,\varphi)=\max_{x\in [0,1]}|\phi(x)-\varphi(x)|\quad \quad\text{ for any }\phi, \varphi \in   C^{0}([0,1]).\end{equation*}

 In \cite{Muentes}, Theorem 4.1 (see also \cite{Carvalho}, Theorem C), it is proved that $$M_{a}([0,1])=\{\phi\in C^{0}([0,1]):  \overline{\text{mdim}}_{\text{M}}([0,1] ,|\cdot |,\phi)=\underline{\text{mdim}}_{\text{M}}([0,1] ,|\cdot|,\phi)=a\}$$ is   dense  in $C^{0}([0,1])$. Following the same steps that in \cite{Muentes}, Theorem 4.1, we prove the next result. 
 
\begin{theorem}
\label{densitypositivemanifoldj}  For any $a\in [0,1]$, the set $$H_{a}([0,1])=\{\phi\in C^{0}([0,1]):  \overline{\emph{mdim}}_{\emph{H}}([0,1] ,|\cdot |,\phi)=\underline{\emph{mdim}}_{\emph{H}}([0,1] ,|\cdot |,\phi)=a\}$$ is   dense  in $C^{0}([0,1])$.
\end{theorem}
  \begin{proof} The set $C^{1}([0,1])$ consisting of all  $C^{1}$-map defined on $[0,1]$ is dense in $C^{0}([0,1])$ and any $\varphi\in C^{1}([0,1])$   has mean Hausdorff dimension equal to 0, because it has finite topological entropy.  Therefore, $H_{0}([0,1])$ is dense in $C^{0}([0,1]).$  
  
Next, fix  $ \phi_{0}\in H_{0}([0,1])$ and let $p$ be a fixed point of $\phi_{0}$. Take  $\varepsilon >0$.   
    Choose  $\delta>0$ such that $|\phi_{0}(x)-\phi_{0} (p)|<\varepsilon/2$ for any $x$ with $|x-p|<\delta$.  Take $\phi_{a}\in H_{a}([0,1])$ for some  $a\in (0,1]$. Set  $J_{1}=[0,p]$, $J_{2}=[p ,p +\delta/2]$, $J_{3}=[p +\delta/2,p +\delta] $ and $J_{4}=[p +\delta,1]$. Set $\psi_{a}\in C^{0}([0,1])$ defined as $$
\psi_{a}(x)= \begin{cases}
    \phi_{0}(x), &  \text{ if }x\in J_{1}\cup J_{4}, \\
   T_{2}^{-1}\phi_{a}T_{2}(x), &  \text{ if }x\in J_{2}, \\
    T_{3}(x), &  \text{ if }x\in J_{3}, 
      \end{cases}
$$ where    $T_{2}:J_{2}\rightarrow I $  is the  affine map such that $T_{2}(p)=0 $ and $T_{2}(p +\delta/2)=1 $,   and  $ T_{3}:J_{3}\rightarrow [ p+\delta/2, \phi_{0}(p+\delta)]$ is the  affine map   such that $ T_{3}(p+\delta/2)=p+\delta/2$  and $ T_{3} (p+\delta)=\phi_{0}(p+\delta)$.  We have that $\hat{d}(\psi_{a},\phi_{0})<\varepsilon.$ Set $A=\cup_{i=0}^{\infty}\psi_{a}^{-i}(J_{2})$ and $B=A^{c}$. We have that $$\text{mdim}_\text{H}(A ,|\cdot |,\psi_{a}|_{A})=\text{mdim}_\text{H}(J_{2} ,|\cdot |,\psi_{a}|_{J_{2}}), \quad \text{mdim}_\text{H}(B ,|\cdot |,\psi_{a}|_{B})=0$$ and hence  \begin{align*}\text{mdim}_\text{H}([0,1] ,|\cdot |,\psi_{a})&=\max \{\text{mdim}_\text{H}(A ,|\cdot |,\psi_{a}|_{A})  ,\text{mdim}_\text{H}(B ,|\cdot |,\psi_{a}|_{B})\}\\
&= \text{mdim}_\text{H}(J_{2} ,|\cdot |,\psi_{a}) 
= \text{mdim}_\text{H}(J_{2} ,|\cdot |,\phi_{a}) 
=a.\end{align*} 
  This fact proves the theorem.
 \end{proof}  

Next, in \cite{Carvalho}, Theorem C (see also \cite{Muentes}, Theorem 4.6), it is proved the set consisting of continuous maps $\phi:[0,1]\rightarrow [0,1]$ with upper metric mean dimension equal to 1 is dense in $C^{0}([0,1])$.
We will prove the set $\mathcal{H}_{1}=\{\phi\in C^{0}([0,1]):\overline{\text{mdim}}_{\text{H}}([0,1],|\cdot |,\phi)=1\}$ is residual in $C^{0}([0,1])$. In order to prove this fact   we need to present the next definition.  

\begin{definition}  For any closed interval $J=[a,b]$, let $\hat{J}=[\frac{2a+b}{3},\frac{a+2b}{3}]$, that is, the second third  of $J$.  For $\epsilon\in(0,1) $ and $k\in\mathbb{N}$, we say  that $J=[a,b]$ is a \textit{strong} $(\epsilon, k)$-\textit{horseshoe} 
of a continuous map $\phi:[0,1]\rightarrow [0,1]$ if $|J|>\epsilon$ and it contains  $k$ closed intervals $J_{1}, \dots, J_{k}\subseteq J$,  with $(J_{s})^{\circ}\cap (J_{r})^{\circ}=\emptyset$ for $s\neq r$, such that  $|J_{i}|>\frac{|J|}{2{k}} $ and  $J\subset  (\phi(\hat{J}_{i}))^{\circ}$ for any $i=1,\dots, k$. \end{definition}
 
 For $\epsilon >0$ and $k\in\mathbb{N}$, set 
\begin{itemize} \item $H(\epsilon, k)=\{\phi\in C^{0}([0,1]):  \phi \text{ has a strong }(\epsilon,k) \text{-horseshoe} \};$ 
\item $H(k)=\bigcup_{i\in\mathbb{N}}H\left( \frac{1}{ i^{2}},3^{k\,i} \right);$ 
 \item  $\mathcal{H} =\overset{\infty }{\underset{k=1}{\bigcap}} H(k).$  \end{itemize}  
 
 \begin{lemma}\label{wwcsfxxc}$\mathcal{H}$ is residual in $C^{0}([0,1])$. \end{lemma}
\begin{proof}
  It is clear that, for any $\epsilon \in (0,1)$ and $k\in\mathbb{N}$, the  set 
$ H(\epsilon, k)$ is not empty (we can construct a continuous function on $ H(\epsilon, k)$ similarly to what was done in Theorem \ref{densitypositivemanifoldj}). 

$ H(\epsilon, k)$  is    open in $C^{0}([0,1])$: if $\phi\in H(\epsilon, k)$ and $J$ is a strong $(\epsilon,k)$-horseshoe of $\phi$ we can take a small enough open neighborhood $U$ of $\phi$ such that for any $\psi \in U$ we have  the same $J$ is a strong $(\epsilon,k)$-horseshoe of $\psi$. 

$H(k)$ is dense in $C^{0}(N)$:  fix $\psi\in C ^{0}(N)$ with a $s$-periodic    point. Every small  neighborhood of the orbit of this point can be perturbed in order to obtain a strong $\left(\frac{1}{i^{2}},3^{ k\, i}\right)$   horseshoe for   a $\phi$ close to $ \psi$ for a large enough $i$. 

The above facts prove that 
$\mathcal{H}$ is residual.\end{proof}
 
  \begin{theorem}\label{teoresjjjjidual} For any $\phi\in \mathcal{H}$ we have $\overline{\emph{mdim}}_{\emph{H}}([0,1], |\cdot |,\phi)=1$. Therefore,  $$\mathcal{H}_{1}=\{\phi\in C^{0}([0,1]):\overline{\emph{mdim}}_{\emph{H}}([0,1], |\cdot |,\phi)=1\}$$ contains a residual subset of  $C^{0}([0,1])$. 
\end{theorem}
\begin{proof}   Take $\phi\in \mathcal{H}$. We have $\phi \in H(k)$ for any $k\geq 1$. Therefore, for any $k\in\mathbb{N}$, there exists   $i_{k}$, with $i_{k}<i_{k+1}$,  such that $\phi $ has a strong  $\left( \frac{1}{i_{k}^{2}},3^{k\, i_{k}} \right)$-horseshoe $J_{i_{k}}$, consisting of $3^{k i_{k}}$ intervals $J_{i_{k}}^{1},\dots, J_{i_{k}}^{3^{ki_{k}}}$, such that $J_{i_{k}}\subset (\phi(\hat{J}_{i_{k}}^{t}))^{\circ}$ for each $t=1,\dots, 3^{ k i_{k}}$. Without loss of generality, we can assume that $J_{i_{k}}= \phi({J}_{i_{k}}^{t}) $ for each $t=1,\dots, 3^{k i_{k}}$.   For any $k\in\mathbb{N}$, set $$\varepsilon_{k}=\max \{|{J}_{i_{k}}^{t}|: t=1,2,\dots, 3^{ki_{k}}\}.$$    
We can divide each $J_{i_k}^{t}$ into $3^{nki_{k}}$ subintervals, $\{ T_{1}^{t},\dots, T_{3^{nki_{k}}}^{t}\}$, such that, for each $l=1,\dots, 3^{ki_{k}}$, there exist  $m_{1},\dots, m_{3^{(n-1)ki_{k}}}\in \{1,\dots, 3^{nki_{k}}\}$ with  $$\phi^{n}(T_{m_h}^{t})=J_{i_k}^{l},\quad\text{for each }t=1,\dots, 3^{ki_{k}},h=1,\dots, 3^{(n-1)ki_{k}}.$$ Hence, for each $m=1,\dots, 3^{nki_{k}}$, we have 
$$ \text{diam}_{d_{n}}(T_{m}^{t})=\varepsilon_{k}\quad\text{ and therefore }\quad\text{H}_{\varepsilon_{k}}^{s} (I_{k},d_{n},\phi|_{I_{k}})\leq \sum_{t=1}^{3^{ki_{k}}}\sum_{i=1}^{3^{nki_{k}}}( \text{diam}_{d_{n}}(T_{i}^{t}))^{s}=3^{(n+1)ki_{k}}(\varepsilon_{k})^{s}.$$
Now, give that $\phi(J_{i_k}^{t})=J_{i_k}$, we have $|\phi(J_{i_{k}}^{t})|=3^{ki_{k}}|J_{i_k}^{t}|=3^{ki_{k}}\varepsilon_{k}.$ Hence, for each $n\in\mathbb{N}$, we need at least $3^{nki_{k}}$ subintervals with diameter less that $\varepsilon_{k}$ to cover each $J_{i_k}^{t}$ and, therefore, $3^{(n+1)ki_{k}}$ subintervals with diameter less that $\varepsilon_{k}$ to cover each $J_{i_k}$. Since every $J_{m}^{t}$ has the same $d_{n}$-diameter, we have  $$\text{H}_{\varepsilon_{k}}^{s} (J_{i_k},d_{n},\phi|_{J_{i_k}})=3^{(n+1)ki_{k}}(\varepsilon_{k})^{s}.$$
Next, if $s\leq -\frac{(n+1)\log 3^{ki_{k}}}{\log(\varepsilon_{k})}$, then $3^{(n+1)ki_{k}}(\varepsilon_{k})^{s}\geq 1$. Thus, 
$$  \text{dim}_{\text{H}}(J_{i_{k}},d_{n},\phi,\varepsilon_{k})=\frac{(n+1)\log (3^{ki_{k}})}{\log(\frac{1}{\varepsilon_{k}})}.$$
Consequently, 
\begin{align*}\overline{\text{mdim}}_{\text{H}}([0,1] ,|\cdot |,\phi) &\geq  \underset{k\rightarrow \infty}{\limsup}\left(\liminf_{n\rightarrow \infty}\frac{1}{n}\frac{(n+1)\log (3^{ki_{k}})}{\log(\frac{1}{\varepsilon_{k}})}  \right)=\underset{k\rightarrow \infty}{\limsup}\left( \frac{\log (3^{ki_{k}})}{\log(\frac{3^{ki_{k}}}{|I_{k}|})}  \right)\\
&=  \underset{k\rightarrow \infty}{\limsup}\frac{ 1}{\left|1-\frac{\log |J_{i_k}|}{\log 3^{ki_{k}}}\right|},\end{align*}
 which proves the lemma.  
\end{proof}
  
 \section{Homeomorphisms on manifold with positive mean Hausdorff dimension}

Suppose that $\text{dim}(N)=m\geq 2$ and fix $\alpha  \in [0,m]$.     In \cite{Carvalho} and \cite{VV}    are presented techniques  to obtain  homeomorphisms  $\phi:N\rightarrow N$ with upper metric mean dimension equal to $\text{dim}(N)$.   In \cite{MuentesA}, Definition 2.1, was   constructed   an $m$-dimensional   $(2k+1)^{m-1}$-horseshoe in order to obtain a homeomorphism on the cube $  [a,b]^{m}$ with upper and lower metric mean dimension equal to $\alpha$.  In the next definition, we will make an appropriate modification to that definition in order to obtain a homeomorphism   on the cube $  [a,b]^{m}$ with lower and upper   mean Hausdorff dimension equal to $\alpha$.     

   \begin{figure}[hbtp]
 \centering
   {\includegraphics[scale=.28]{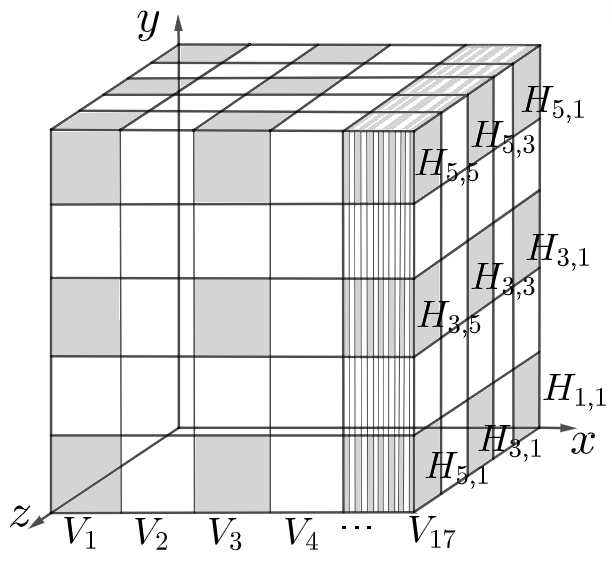}}
 \caption{3-dimensional  9-horseshoe}\label{mm}
 \end{figure}
 
 \begin{definition}[$m$-dimensional   $(2k+1)^{m-1}$-horseshoe]\label{3dimensionalhorse}  
 Fix $m\geq 2$. Take $E=[a,b]^{m}$ and set $|E|=b-a$. For a fixed natural number $k>1$,   take the sequence    $a=t_0<t_{1}<\cdots<t_{4k}<t_{4k+1}=b $, with $|t_{i}-t_{i-1}|=\frac{b-a}{4k+1}$,
 and  consider   $$H_{i_{1},i_{2},\dots,i_{m-1}}=[a,b]\times[t_{i_{1}-1},t_{i_{1}}]\times  \cdots \times[t_{i_{m-1}-1}, t_{i_{m-1}}],\quad \text{for }i_{j}\in\{1,\dots, 4k+1\}.$$   Take $a=s_0<s_{1}<\cdots<s_{2(2k+1)^{m-1}-2}<s_{2(2k+1)^{m-1}-1}=b $ 
 and consider     $$ V_{l}=[s_{l-1}, s_{l}]\times [a,b]^{m-1},\quad \text{for } l=1, 2, \dots, 2(2k+1)^{m-1}-1.$$ 
 We will require that:
 $$ |s_{l}-s_{l-1}|=\frac{b-a}{4k+1}\quad\text{for } l=1,\dots,4k\quad\text{and}\quad  |s_{a}-s_{a-1}| =|s_{b}-s_{b-1}| \quad\text{for } a,b\geq4k+1.
 $$ 
 See Figure \ref{mm}. We say that $E\subseteq A\subseteq \mathbb{R}^{m}$ is an     $m$-\textit{dimensional       $(2k+1)^{m-1}$-horseshoe} for a homeomorphism $\phi:A\rightarrow A$ if:   
\begin{itemize} 
\item $\phi(a,a,\dots,a,b)=(a,a,\dots,a,b)$ and $\phi(b,b,\dots,b,a)=(b,b,\dots,b,a)$;
\item For  any $H_{i_{1},i_{2},\dots,i_{m-1}}$, with $i_{j}\in\{1,3,\dots, 4k+1\}$, there exists  some   $l\in\{1,3,\dots, 2(2k+1)^{m-1}-1\}$ with $$\phi(V_{l})=H_{i_{1},i_{2},\dots,i_{m-1}}  \quad\text{and}\quad \phi|_{V_{l}}:V_{l}\rightarrow H_{i_{1},i_{2},\dots,i_{m-1}} \quad \text{is linear}.$$ 
\item For any $l=2,4,\dots, 2(2k+1)^{m-1}-2$,  $\phi(V_{l})\subseteq A\setminus E$.\end{itemize} 
  \end{definition}   
  
Note that $$\tilde{V}_{4k+1}:=\bigcup_{l=4k+1}^{2(2k+1)^{m-1}-1}V_{l}=\left[b-\frac{b-a}{4k+1},b\right]\times [a,b]^{m-1}.$$

\begin{lemma}\label{lemma2} Let $E^{\prime}:=[a^{\prime},b ^{\prime}]^{m}$ and $  E:=[a,b]^{m}\subseteq \mathbb{R}^m$ be closed $m$-cubes  with $E^{\prime} \subsetneqq (E)^{\circ}$ and fix any $k\in\mathbb{N}$. There   exists a homeomorphism $\psi: E\longrightarrow E$ such that $\psi|_{E^{\prime}}:  E^{\prime}\longrightarrow E^{\prime}$ is an $m$-dimensional $(2k+1)$-horseshoe, and   $\phi:=\psi^{2}$ satisfies $\phi|_{\partial E }\equiv Id$ and $h_{\emph{top}}(\phi)=h_{\emph{top}}(\phi|_{E^{\prime}\cap \phi(E^{\prime} )})$, where $\partial A$ is the boundary of a subset $A$ and $Id$ is the identity map.\end{lemma}
\begin{proof}
    For the sake of clarity in both presentation and illustration, we will demonstrate the lemma specifically for the case where $m=2$, although it is important to note that the same principle applies for any $m\geq 2$.
    
    Suppose that $\psi:E \rightarrow E $ has an $m
    $-dimensional $(2k+1)$-horseshoe $E^{\prime}$. In Figure \ref{egefscs} we show the the image $\phi(E^{\prime})=\psi^{2}(E^{\prime})$ of an 2-dimensional 3-horseshoe for $\psi.$ We will extend $\psi$ with the conditions in the lemma.  
   \begin{figure}[hbtp]
 \centering
  { \subfigure[$\phi(E)=\psi^{2}(E)$]{\includegraphics[scale=.26]{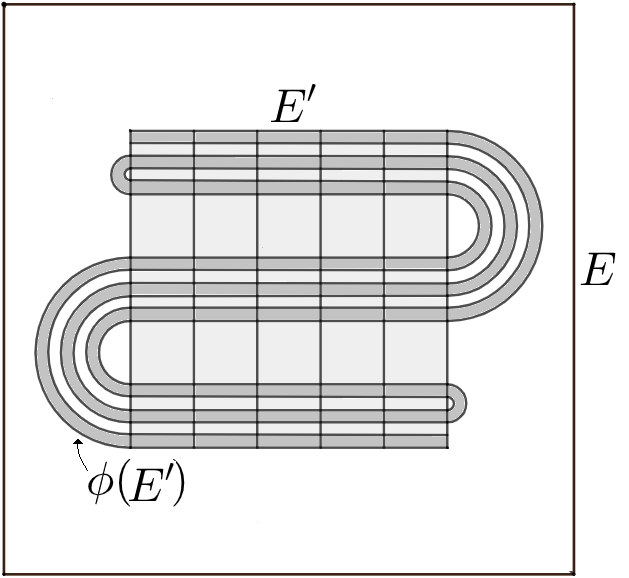}\label{egefscs}}  }  \subfigure[$K_{1},\dots, K_{4}$]{\includegraphics[scale=.26]{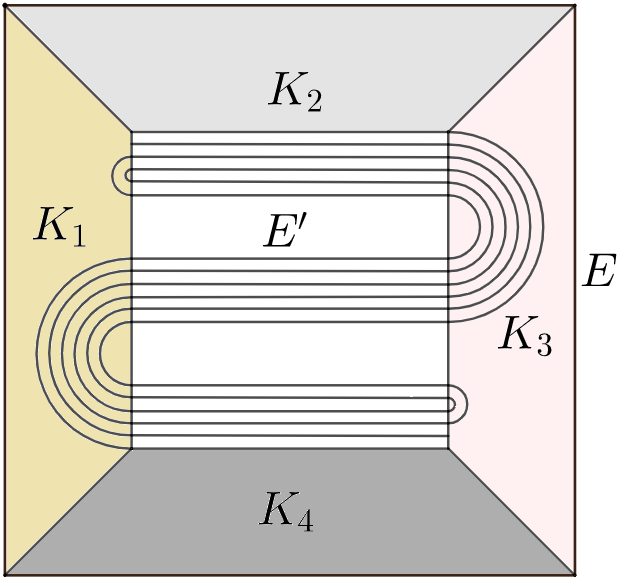}\label{egefsdcs}}  \subfigure[$Q_{1},\dots, Q_{4}$]{\includegraphics[scale=.26]{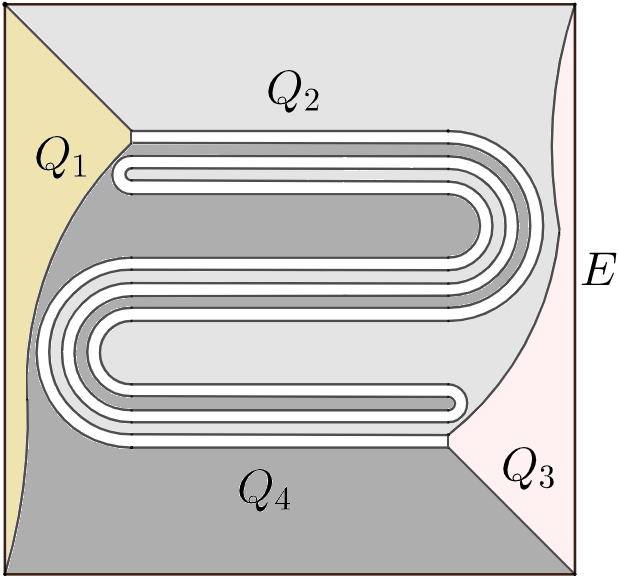}\label{essdgefscs}}
 \caption{3-dimensional  9-horseshoe}\label{wdqdmm}
 \end{figure}
 Next,
\begin{itemize}
    \item let $K_{1}$ be the polygonal region with vertices  $(a , a ),(a^{\prime}, a^{\prime}),(a^{\prime}, b^{\prime}),(a,b) ;$
   \item let $K_{2}$ be the polygonal region with vertices  $(a , b ),(a^{\prime}, b^{\prime}),(b^{\prime}, b^{\prime}),(b,b) ;$
   \item let $K_{3}$ be the polygonal region with vertices  $(b , b ),(b^{\prime}, b^{\prime}),(b^{\prime}, a^{\prime}),(b,a) ;$
   \item let $K_{4}$ be the polygonal region with vertices  $(a , a ),(a^{\prime}, b^{\prime}),(b^{\prime}, a^{\prime}),(b,a) .$
\end{itemize}
These regions are shown in Figure \ref{egefsdcs}. Now,
\begin{itemize}
    \item let $Q_{1}$ be  region bounded by  the next  curves: a curve $L_{1}$ from  $(a , a )$ to $ \phi(a^{\prime}, a^{\prime})$ contained in $K_{1}$ and such that   does not intersect $\phi(E)$;  the segment from  $\phi(a^{\prime}, a^{\prime})$ to $(a^{\prime}, b^{\prime}) ;$ the segment from  $(a^{\prime}, b^{\prime})$ to $(a,b) ;$ the segment from  $(a , b )$ to $(a,a) .$ 
   \item let $Q_{2}$  be the  region bounded by  the next  curves: the segment from  $(a , b )$ to $  (a^{\prime}, b^{\prime})$; $\phi(\{b^{\prime}\}\times [a^{\prime},b^{\prime}])$; a curve $L_{2}$ from  $\phi(b^{\prime} , b^{\prime} )$ to $  (b,b)$ contained in $K_{3}$ and such that   does not intersect $\phi(E)$;  the segment from  $(b,b)$ to $(a, b) .$ 
   \item let $Q_{3}$  be the  region bounded by  the next  curves: $L_{2}$; the segment from  $(b , b )$ to $  (b, a)$; the segment from  $(b,a)$ to $( b^{\prime}, a^{\prime}) $;   the segment from $( b^{\prime}, a^{\prime})$ to $\phi(b^{\prime} , b^{\prime} )$. 
   \item let $Q_{4}$  be the region bounded by  the next  curves:  the segment from  $(b,a)$ to $( b^{\prime}, a^{\prime}) $; $\phi(\{a^{\prime}\}\times [a^{\prime},b^{\prime}])$; $L_{1}$; the segment from $(a,a)$ to $(b,a)$. 
\end{itemize}
These regions are shown in Figure \ref{essdgefscs}. The annular regions $K=K_{1}\cup K_{2}\cup K_{3}\cup K_{4}$ and $Q=Q_{1}\cup Q_{2}\cup Q_{3}\cup Q_{4}$ are homeomorphic. We can consider a homeomorphism $\varphi:K\to Q$ such that $\varphi(K_{i})=Q_{i}$, $\varphi|_{\partial E} =\phi|_{\partial E}  $,  and $\varphi|_{\partial E^{\prime}} =Id|_{\partial E^{\prime}} $. Therefore, we can to extend $\phi$ to $E$,  setting $\phi|_{K}=\varphi$.  

Next, notice that $h_{\text{top}}(\phi)=h_{\text{top}}(\phi|_{E^{\prime}\cap \phi(E^{\prime} )})$, because $\Omega (\phi)\subseteq  (E^{\prime}\cap \phi(E^{\prime} ))\cup \partial E $ and $\phi|_{\partial E}$ is the identity. The above facts prove the lemma. 
\end{proof}

   On any  $m$-cube  $E\subseteq\mathbb{R}^{m}$, we will consider the metric $\rho(\cdot,\cdot)$, given by $$\rho((x_{1},\dots,x_{m}),(y_{1},\dots,y_{m}))= \max\{|x_{1}-y_{1}|,\dots,|x_{m}-y_{m}|\}.$$

\begin{lemma}\label{lemma22}  Let $\phi:  {\mathcal{C}}:=[0,1]^{m}\rightarrow  {\mathcal{C}}$ be a homeomorphism,  $E_{k}=[a_k , b_k]^{m}$ and  $E_{k}^{\prime}$   sequences of cubes such that:
\begin{itemize}\item[C1.] $ E_{k}\subset (E_{k}^{\prime})^{\circ}$ and $(E_{k}^{\prime})^{\circ}\cap (E_{s}^{\prime})^{\circ}=\emptyset$ for $k\neq s$.
\item[C2.]   $S:=\cup_{k=1}^{\infty}E^{\prime}_{k}\subseteq {\mathcal{C}}$. 
\item[C3.] each $E_{k}$ is an $m$-dimensional   $3^{k(m-1)}$-horseshoe for $\phi$;     
\item[C4.] For each $k$, $\phi|_{E_{k}^{\prime}}: E_{k}^{\prime}\rightarrow E_{k}^{\prime}$ satisfies the properties in Lemma \ref{lemma2};
\item[C5.] $\phi|_{{\mathcal{C}}\setminus S}: {\mathcal{C}}\setminus S\rightarrow  {\mathcal{C}}\setminus S$ is the identity. 
\end{itemize}
 We have:
 \begin{itemize}\item[(i)]  For a fixed $r\in (0,\infty)$, if $|E_{k}|=\frac{B}{3^{kr}}$ for each $k\in\mathbb{N}$, where $B>0$ is a constant, then $$  \emph{mdim}_{\emph{H}}({\mathcal{C}},\rho,\phi^{2}) =\frac{m}{r+1} . $$
 \item[(ii)] If $|E_{k}|=\frac{B}{k^{2}}$ for each $k\in\mathbb{N}$, where $B>0$ is a constant, then $$  \emph{mdim}_{\emph{H}}( {\mathcal{C}},\rho,\phi^{2}) =m . $$
 \end{itemize}\end{lemma}
\begin{proof}  We will prove (i), since (ii) can be proved  analogously. Set $\varphi=\phi^{2}$. Note that 
$$  {{\text{mdim}}_{\text{H}}}( {\mathcal{C}},\rho,\varphi ) ={{\text{mdim}}_{\text{H}}}(S,\rho,\varphi|_{S} ). $$

Take any $\varepsilon\in (0,1)$. For any $k\geq 1$, set $ \varepsilon_k=  \frac{|E_{k}|}{2(3^{k})-1}  =\frac{B}{(2(3^{k})-1)3^{kr}} $. There exists $k\geq 1$ such that  $\varepsilon \in [\varepsilon_{k+1}, \varepsilon_{k}]$. We have  
 \begin{equation*}     \text{H}_{\varepsilon_{k}}^{s}( N,\rho_{m},\varphi) \geq   \text{H}_{\varepsilon_{k}}^{s}( E_{k}\cap \Omega (\varphi),\rho_{m},\varphi|_{E_{k}\cap \Omega (\varphi)}) \quad\text{for any }m\geq 1 .  \end{equation*}
Since $E_{k}$ is an $m$-dimensional $3^{k(m-1)}$-horseshoe for $\phi$, for each $k\geq 1$, consider   $${H}^{k}_{i_{1},i_{2},\dots,i_{m-1}}=[a_{k},b_{k}]\times[t_{i_{1}-1},t_{i_{1}}]\times  \cdots \times[t_{i_{m-1}-1}, t_{i_{m-1}}],\quad\text{for }i_{j}= 1,\dots, 2(3^{k}) -1,$$   and  $$  V_{l}^{k}=[s_{l-1}, s_{l}]\times[a_{k},b_{k}]^{m-1},\quad\text{for }l=1,\dots, 2(3^{k})^{m-1}-1,$$ as in Definition \ref{3dimensionalhorse}.  Set $$U_{l}^{k}=V_{l}^{k}\quad\text{for }l=1,\dots, 2(3^{k})-2\quad \text{and}\quad U_{2(3^{k})-1}^{k}=\bigcup_{l=2(3^{k})-1}^{2(3^{k})^{m-1}-1}V_{l}^{k}.$$
For each $j=1,\dots, m$, let $i_{j}\in\{1,\dots,  2(3^{k})-1 \}$   
 and  take $$C^{k}_{i_{1},\dots,i_{m}} =  H^{k}_{i_{1},i_{2},\dots,i_{m-1}}\cap U^{k}_{i_{m}}    .$$  
For $t=1,\dots,n$, let  $i_{1}^{(t)},$ $ \dots,i_{m}^{(t)}\in\{1,3,5,\dots, 2(3^{k}) -1\}$  and      
set   \begin{align*} 
 C^{k}_{ i_{1}^{(2)},\dots, i_{m}^{(2)} ,i_{1}^{(1)},\dots, i_{m}^{(1)}} &=\varphi^{-1}\left[\varphi\left(C^{k}_{  i_{1}^{(2)},\dots, i_{m}^{(2)}}\right)\cap C^{k}_{ i_{1}^{(1)},\dots,i_{m}^{(1)}}\right]\\
  &\vdots\\
 C^{k}_{ i_{1}^{(n)},\dots,i_{m}^{(n)},  \dots,    i_{1}^{(1)},\dots, i_{m}^{(1)}} &=\varphi^{-(n-1)}\left[\varphi^{n-1}\left(C^{k}_{ i_{1}^{(n)},\dots,i_{m}^{(n)}, \dots,  i_{1}^{(2)},\dots,i_{m}^{(2)}}\right)\cap C^{k}_{ i_{1}^{(1)},\dots,i_{m}^{(1)}}\right]
 \end{align*} 
 From the definition of $\varphi$, these sets are non-empty.  
 Furthermore,    $$\text{diam}_{n}\left(C^{k}_{  i_{1}^{(n)},\dots,i_{m}^{(n)},  \dots ,i_{1}^{(1)},   \dots,   i_{m}^{(1)}}\right)=\varepsilon_{k},\quad\text{ if  } {i_{t}^{(j)}\in \{{1},3,5,\dots,2({3^{k}})-1\}},$$   We have   $3^{knm}$ sets of this form. Hence, 
 \begin{equation*}    \text{H}_{\varepsilon_{k}}^{s}( E_{k}\cap \Omega (\varphi),\rho_{n},\varphi|_{E_{k}\cap \Omega (\varphi)})\leq \sum_{t=1}^{3^{knm}}\left(\frac{B}{(2(3^{k})-1)3^{kr}}\right)^{s}=3^{knm}\left(\frac{B}{(2(3^{k})-1)3^{kr}}\right)^{s}.
 \end{equation*}
 We have \begin{equation*}    \text{H}_{\varepsilon_{k}}^{s}( E_{k}\cap \Omega (\varphi),\rho_{n},\varphi|_{E_{k}\cap \Omega (\varphi)}) =3^{knm}\left(\frac{B}{(2(3^{k})-1)3^{kr}}\right)^{s},
 \end{equation*} because $\text{diam}_{n}\left(C_{  i_{1}^{(n)},\dots, i_{m}^{(n)},  \dots ,i_{1}^{(1)},   \dots,  i_{m}^{(1)}}\right)=\varepsilon_{k}$.    Next,  $$3^{knm}\left(\frac{B}{(2(3^{k})-1)3^{kr}}\right)^{s}\geq 1\Longleftrightarrow   3^{knm}  \geq   \left(\frac{(2(3^{k})-1)3^{kr}}{B}\right)^{s}\Longleftrightarrow \frac{\log 3^{knm}}{\log \left(\frac{(2(3^{k})-1)3^{kr}}{B}\right)}\geq s.$$
 Therefore,  \begin{align*}\label{exxample12}   \text{dim}_{\text{H}}(E_{k}\cap\Omega(\varphi),\rho_{n},\varepsilon_{k})=\frac{\log 3^{knm}}{\log \left(\frac{(2(3^{k})-1)3^{kr}}{B}\right)}\end{align*} and hence \begin{align*}\lim_{n\rightarrow \infty}\frac{1}{n}\text{dim}_{\text{H}}(N,\rho_{n},\varepsilon)\geq \lim_{n\rightarrow \infty}\frac{1}{n}\text{dim}_{\text{H}}(E_{k}\cap\Omega(\varphi),\rho_{n},\varepsilon_{k})=\frac{\log 3^{km}}{\log \left(\frac{(2(3^{k})-1)3^{kr}}{B}\right)}.\end{align*} 
 Thus, \begin{equation}\label{nmfwf}
   \text{mdim}_{\text{H}}(N,\rho,\varphi)\geq \lim_{k\rightarrow\infty}\frac{\log 3^{km}}{\log \left(\frac{(2(3^{k})-1)3^{kr}}{B}\right)}=\frac{m}{1+r}.\end{equation} 
 
Next, we will prove that ${ \overline{\text{mdim}}_{\text{M}}}(\mathcal{C} ,\rho,\varphi)\leq \frac{m}{r+1}$. Note that $ \frac{\log 3^{km}}{\log[4(2(3^{k})-1)3^{kr}B^{-1}]} \rightarrow\frac{m}{1+r}$ as $k\rightarrow \infty$. Hence, for any $\delta >0$, there exists $k_{0}\geq 1$, such that,  for any  $k> k_{0}$, we have $ \frac{\log 3^{km}}{\log[4(2(3^{k})-1)3^{kr}B^{-1}]} <\frac{m}{1+r}+\delta$. Hence, suppose that $\varepsilon>0$ is small enough such that $\varepsilon<\epsilon_{k_{0}}$.   Set $ \tilde{\Omega}_{k}=\underset{n\in\mathbb{Z}}{\bigcap}\varphi^{n}(E_{k})$ and take $\tilde{\Omega}=\underset{k\in\mathbb{N}}{\bigcup}\Omega_{k}$.    We have $$  {{\text{mdim}}_{\text{M}}}( {\mathcal{C}},\rho,\varphi ) ={{\text{mdim}}_{\text{M}}}(\tilde{\Omega},\rho,\varphi|_{\tilde{\Omega}} ). $$ 
If $x\in \tilde{\Omega}$, then $x$ belongs to some $C^{k}_{i_{1}^{(n)},\dots,i_{m}^{(n)},  \dots,  i_{1}^{(1)},\dots, i_{m}^{(1)}}$, where ${i_{t}^{(j)}\in \{{1},3,5,\dots,2({3^{k}})-1\}}$. Furthermore, $\rho_{n}(x,y)\leq 4 \varepsilon_{k}$ for any $y\in C^{k}_{i_{1}^{(n)},\dots,i_{m}^{(n)},  \dots,  i_{1}^{(1)},\dots, i_{m}^{(1)}}$. Hence, if $Y_{k}=\cup_{j=1}^{k}E_{j}$, for every $ n\geq1$, we have 
  \begin{align*}    \text{span}(n,\varphi  |_{Y_{k}}, 4\varepsilon) & \leq \sum_{j=1}^{k}   \frac{3^{jnm}}{\varepsilon}  \leq k\frac{3^{knm}}{\varepsilon} .  \end{align*} 
 Therefore,  
 \begin{align*} \frac{\text{span}(\varphi  |_{Y_{k}}  ,4 \varepsilon)}{|\log 4\varepsilon|}& \leq\limsup_{n\rightarrow \infty} \frac{\log\left[  k\frac{3^{knm}}{\varepsilon} \right]}{n |\log  4\varepsilon_{k}|}  = \frac{\log 3^{km}}{\log[4(2(3^{k})-1)3^{kr}B^{-1}]} <\frac{m}{1+r}+\delta.
 \end{align*}
This fact implies that for any $\delta >0$ we have \begin{align*}  {\overline{\text{mdim}}_{\text{M}}}(S ,\rho,\varphi|_{S} )<\frac{m}{r+1}+\delta\quad\text{and hence }\quad {\overline{\text{mdim}}_{\text{M}}}(S ,\rho,\varphi|_{S})\leq\frac{m}{r+1}.\end{align*} Hence,  ${ \overline{\text{mdim}}_{\text{M}}}(\mathcal{C} ,\rho,\varphi)\leq \frac{m}{r+1}$. 
 From \eqref{mnqw} we have \begin{equation}\label{mbdshjfw}{\text{mdim}}_{\text{H}}(\mathcal{C} ,\rho,\varphi)\leq \overline{\text{mdim}}_{\text{M}}(\mathcal{C} ,\rho,\varphi),\end{equation} therefore, from \eqref{nmfwf} and  \eqref{mbdshjfw}  it follows that $ {\text{mdim}}_{\text{M}}(\mathcal{C} ,\rho,\varphi)={\text{mdim}}_{\text{H}}(\mathcal{C} ,\rho,\varphi)=\frac{m}{r+1}.$ 
\end{proof} 
  
   On $\text{Hom}(N)$ we will consider the metric \begin{equation*}\label{cbenfn333} \hat{d}(\phi,\varphi)=\max_{p\in N}\{d(\phi(p),\varphi(p)),d(\phi^{-1}(p),\varphi^{-1}(p))\}\quad \quad\text{ for any }\phi, \varphi \in   \text{Hom}(N).\end{equation*}

 Take $\alpha\in[0,m]$ and set $$H_{\alpha}(N)=\{\phi :N\rightarrow N\in \text{Hom}(N):  \underline{\text{mdim}}_{\text{H}}(N,d,\phi)=\overline{\text{mdim}}_{\text{H}}(N,d,\phi)=\alpha\}. $$ In \cite{MuentesA}, Theorem 3.1, is proved the set consisting of homeomorphisms $\phi:N\rightarrow N$ with upper and lower metric mean dimension equal to  $\alpha$ is dense in $\text{Hom}(N)$. We  adapt that proof in order to show the analog result for mean Hausdorff dimension.

\begin{theorem}
\label{densitypositivemanifold}  For any $\alpha\in [0,m]$, the set $H_{\alpha}(N)  $ is   dense  in $\emph{Hom}(N)$.
\end{theorem}
\begin{proof}  For each $p\in N$, consider the      \textit{exponential map}  $$\text{exp}_{p}: B_{\delta^{\prime}}(0_{p})\subseteq T_{p}N\rightarrow B_{\delta^{\prime}}(p)\subseteq N,$$ where $0_{p}$ is the origin in the tangent space $T_{p}N$,  $\delta^{\prime}$ is the \textit{injectivity radius} of $N$ and      $B_{\epsilon}(x)$ denote the open ball of radius $\epsilon>0$ with center $x$. 
 We will take $ {\delta}_{N}=\frac{\delta^{\prime}}{2}$.   
 
The set  $P^{r}(N)$ consisting of all $C^{r}$-diffeomorphisms on $N$ with a periodic point is $C^{0}$-dense in $\text{Hom}(N)$ (see \cite{Artin},    \cite{Hurley}).  Hence, in order to prove the theorem, it is sufficient to show that $H_{\alpha}(N)$ is dense in $P^r(N)$.    
Fix $\psi\in P^{r}(N)$
 and take any  $\varepsilon\in (0,\delta_{N})$. Suppose that $p\in N$ is a periodic point of $\psi$, with period $k$. Take   $\beta>0$,  small enough, such  that $$[-\beta, \beta]^{m}\subset  B_{\frac{\varepsilon}{2}}(0)\,\text{ and }\, C_{i+1}:= \psi(\exp_{\psi^{i}(p)}((-\beta,\beta)^{m})))\subset D_{i+1}:=\exp_{\psi^{i+1}(p)}(B_{\frac{\varepsilon}{2}}(0_{\psi^{i+1}(p)})) ,$$  for each $i=0,\dots, k-1.$ Notice that $$p=\psi^{k}(p)=\text{exp}_{p}(0_{p})\in C_{k}.$$ Now, let $\lambda\in(0,\frac{\beta}{4})$ be such that  $$ E_{i+1}:=\exp_{\psi^{i+1}(p)} ([-\lambda,\lambda]^{m})\subset C_{i+1}, \quad\text{ for each }i=0,\dots, k-1$$  \begin{figure}[hbtp]
 \centering
   {\includegraphics[scale=.5]{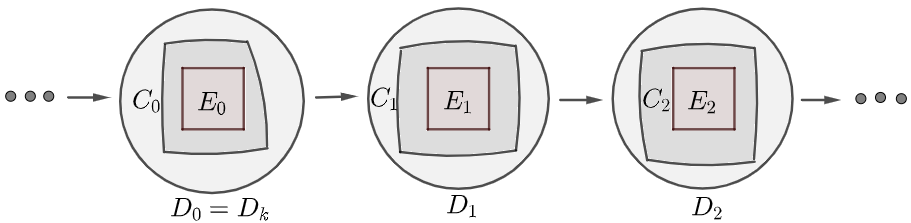}}
 \caption{$E_{i}\subset C_{i}\subset D_{i}$.}\label{mm1}
 \end{figure} (see Figure \ref{mm1}).  The annular regions 
$$[-\beta,\beta]^{m}\setminus (-\lambda, \lambda)^{m} \quad\text{ and }\quad \exp^{-1}_{\psi^{i+1}(p)}(C_{i+1}\setminus E_{i+1}^{\circ}) $$ are homeomorphic. Consider a homeomorphism $$H_i:   [-\beta,\beta]^{m}\setminus (-\lambda, \lambda)^{m} \to \exp^{-1}_{\psi^{i+1}(p)}(C_{{i+1}}\setminus E_{i+1}^{\circ}),$$ such that 
\[H_{i}|_{\partial [-\lambda,\lambda]^{m}}=Id\ \ \text{and}\ \ H_{i}|_{\partial [-\beta,\beta]^{m}}=\exp^{-1}_{\psi^{i+1}(p)}\circ \psi\circ \exp _{\psi^{i}(p)} .\]
Set $\varphi_{\alpha}: N\rightarrow N$, given by 
\[\varphi_{\alpha}(q)=\begin{cases}\exp_{\psi^{i+1}(p)}H_i(\exp^{-1}_{\psi^{i}(p)}(q)), &\text{if }q\in R_i:=\exp_{\psi^{i}(p)}  \left([-\beta,\beta]^{m}\setminus (-\lambda,\lambda)^{m} \right)\\
\exp_{\psi^{i+1}(p)}\phi_{\alpha}(\exp_{\psi^{i}(p)}^{-1}(q)), &\text{if }q\in Q_i:=\exp_{\psi^{i}(p)}  \left([ -\lambda,\lambda]^{m} \right)\\
\psi(q),  & \text{otherwise}, \end{cases}\]
where  $\phi_{\alpha}\colon[ -\lambda,\lambda]^{m} \to [  -\lambda,\lambda]^{m}$ is a homeomorphism     which satisfies  $\phi_{\alpha}|_{\partial [ -\lambda,\lambda]^{m} }=Id$ and $$  {{\text{mdim}}_{\text{H}}}([ -\lambda,\lambda]^{m}  ,  \rho,\phi_{\alpha})=\alpha,$$ (see Lemma \ref{lemma22}). 
Set $K:=\bigcup_{i=0}^{k-1}Q_{i}$. Note that $N\setminus K$ is $\varphi_{\alpha}$ invariant and     $$ \text{mdim}_\text{H}(N\setminus K ,d,\varphi_{\alpha}|_{N\setminus K})=0.$$
 Hence,   $$ \text{mdim}_\text{H}(N ,d,\varphi_{\alpha})=  \text{mdim}_\text{H}(K  ,d,\varphi_{\alpha}|_{K}).$$ 
   Note that if $q\in Q_{i}$, we have
$$  (\varphi_{\alpha})^{s}(q) =   \text{exp}_{\psi^{(i+s)\text{ mod }k}(p)}\circ(\phi_{\alpha})^{s}\circ\text{exp}^{-1}_{\psi^{i}(p)}(q)  \quad\text{for any }s\in\mathbb{N}.$$
 Hence,  $D\subseteq [-\lambda,\lambda]^{m}$   has $\rho_{n}$-diameter  less than  $\epsilon\in(0,\delta_{N})$ if and only if $\text{exp}_{\psi^{i}(p)}(D)\subseteq Q_{i}$   has $d_{n}$ diameter   less than  $\epsilon\in(0,\delta_{N})$ for all $i=0,\dots, k-1$. 
Since $K=\bigcup_{i=0}^{k-1}Q_{i}$, we have   for any $j\in \{0,\dots,k-1\}$ that 
  \begin{align*}  \text{mdim}_\text{H}( N ,d,\varphi_{\alpha})&=\text{mdim}_\text{H}( K ,d,\varphi_{\alpha}|_{K})= \underset{i=0,\dots,k-1}{\max}\{\text{mdim}_\text{H}( Q_{i},d,\varphi_{\alpha} )\}\\
  &= \text{mdim}_\text{H}( Q_{j},d,\varphi_{\alpha})= \text{mdim}_\text{H}([-\lambda,\lambda]^{n} ,\rho,\phi_{\alpha})=\alpha.\end{align*}   
 It is clear that $\hat{d}(\varphi_{\alpha},\psi)<\varepsilon$, which proves the theorem.
\end{proof}

 \section{Tipical homeomorphism on a manifold has maximal mean Hausdorff   dimension}\label{Section6}

Suppose $\text{dim}(N)=m\geq 2$, and let $\alpha \in [0,m]$ be fixed. In \cite{Carvalho}, it was demonstrated that the collection of all homeomorphisms on $N$ with an upper metric mean dimension equal to $m$ forms a residual set in $\text{Hom}(N)$ (also see \cite{MuentesA} and \cite{VV}). The case for continuous maps (which need not be homeomorphisms) was established in \cite{Muentes}. In this section, we will establish that the set of all homeomorphisms on $N$ with an upper mean Hausdorff dimension equal to $m$ constitutes a residual set in $\text{Hom}(N)$.

 \begin{definition}[$m$-dimensional    strong horseshoe]\label{3ddimensionalhorse} Let $E=[a,b]^{m}$ and set $|E|=b-a$. For a fixed natural number $k>1$, set  $\delta_{k}=  \frac{b-a}{4k+1}$.  For $i=0,1, 2, \dots, 4k+1$, set $t_i= a+i\delta_{k} $ and consider   $$H_{i_{1},i_{2},\dots,i_{m-1}}=[a,b]\times[t_{i_{1}-1},t_{i_{1}}]\times  \cdots \times[t_{i_{m-1}-1}, t_{i_{m-1}}],\quad \text{for }i_{j}\in\{1,\dots, 4k+1\}.$$ 
 Take $a=s_0<s_{1}<\cdots<s_{2(2k+1)^{m-1}-2}<s_{2(2k+1)^{m-1}-1}=b $ 
 and consider     $$ V_{l}=[s_{l-1}, s_{l}]\times [a,b]^{m-1},\quad \text{for } l=1, 2, \dots, 2(2k+1)^{m-1}-1.$$ 
 We will require that:
 $$ |s_{l}-s_{l-1}|=\frac{b-a}{4k+1}\quad\text{for } l=1,\dots,4k\quad\text{and}\quad  |s_{a}-s_{a-1}| =|s_{b}-s_{b-1}| \quad\text{for } a,b\geq4k+1.
 $$ 
  We say that $E\subseteq A\subseteq \mathbb{R}^{m}$ is an     $m$-\textit{dimensional   strong   $(\epsilon,(2k+1)^{m-1})$-horseshoe} for a homeomorphism $\phi:A\rightarrow A$ if $|E|>\epsilon$ and furthermore:  
\begin{itemize} \item For  any $H_{i_{1},i_{2},\dots,i_{m-1}}$, with $i_{j}\in\{1,3,\dots, 4k+1\}$, there exists     $l\in\{1,3,\dots, 2(2k+1)^{m-1}-1\}$ with $$H_{i_{1},i_{2},\dots,i_{m-1}}\subseteq \phi(V_{l})^{\circ}.$$ 
\item For any $l=2,4,\dots, 2(2k+1)^{m-1}-2$,  $\phi(V_{l})\subseteq A\setminus E$.\end{itemize} 
  \end{definition}

\begin{definition} Let $N$ be an $m$-dimensional Riemannian manifold and fix $k\geq 1$. We say that $\phi\in \text{Hom}(N)$ has an  $m$-\textit{dimensional strong} $( \epsilon, (2k+1)^{m-1})$-\textit{horseshoe} $E=[a,b]^{m}$,  if there is $s$ and an exponential charts $\text{exp}_{i}: B_{\delta_{N}}(0)\rightarrow N$, for $i=1,\dots,s$,   such that: \begin{itemize} \item $\phi_{i}=\text{exp}_{(i+1)\text{mod}\,  s}\circ \phi \circ \text{exp}^{-1}_{i}: B_{\delta}(0 )\rightarrow B_{\delta_{N}}(0)$ is well defined for some $\delta \leq \delta_{N}$;
 \item $E$ is an $m$-dimensional  strong $(\epsilon, (2k+1)^{m-1})$-horseshoe for $\phi_{i} $. \end{itemize}  
 To simplify the notation, we will set $\phi_{i}=\phi$ for each $i=1,\dots, s$. \end{definition}
 
\noindent  For $\epsilon >0$ and $k\in\mathbb{N}$, we consider the sets 
\begin{itemize} \item $G(\epsilon, k)=\{\phi^{2} \in \text{Hom}(N):  \phi \text{ has a strong }(\epsilon,k) \text{-horseshoe} \}  ;$ 
\item $G(k)=\bigcup_{i\in\mathbb{N}}H\left( \frac{1}{ i^{2}},3^{ki(m-1)} \right);$ 
 \item $  \mathcal{G}=\overset{\infty }{\underset{k=1}{\bigcap}} G(k).$  \end{itemize}  

A similar subset is defined in  \cite{Muentes}, Section 4, and it is proved that is residual and all homeomorphism in that set has metric mean dimension equal to $m$ (see Lemma 4.4 in \cite{Muentes}). We can to adapt that proof in order to show that $\mathcal{G}$ is residual and furthermore we prove that all homeomorphism in $\mathcal{G}$ has mean Hausdorff dimension equal to $m$. First, we can show the following lemma analogously to Lemma \ref{wwcsfxxc}, and therefore we omit the proof. 

  \begin{lemma}\label{ffeaqw}
  $\mathcal{G} $ is residual in $\emph{Hom}(N)$.    
  \end{lemma}

  \begin{theorem}\label{teoeeeresjjjjidual} For any $\phi\in \mathcal{G}$ we have $\overline{\emph{mdim}}_{\emph{H}}(N,d,\phi)=m$. Therefore,  $$\mathcal{G}_{m}=\{\phi\in \emph{Hom}(N):\overline{\emph{mdim}}_{\emph{H}}(N,d,\phi)=m\}$$ contains a residual subset of  $\emph{Hom}(N)$. 
\end{theorem}
\begin{proof}  
 We prove  that $\overline{\text{mdim}}_{\text{H}}(N,d,\phi)=m$ for any $\phi\in \mathcal{G}$.  For this sake, take $\varphi=\phi^{2}\in \mathcal{G}$. We have $\varphi \in G(k)$ for any $k\geq 1$. Therefore, for any $k\in\mathbb{N}$, there exists   $i_{k}$, with $i_{k}<i_{k+1}$,  such that $\phi $ has a strong  $\left( \frac{1}{i_{k}^{2}},3^{m\,k\, i_{k}} \right)$-horseshoe $E_{{k}}=[a_{k},b_{k}]^{m}$, such that $|E_{{k}}|>\frac{1}{i_{k}^{2}}$, consisting of   rectangles  $$H^{i_{k}}_{j_{1},j_{2},\dots,j_{m-1}}=[a_{k},b_{k}]\times[t^{{k}}_{j_{1}-1},t^{{k}}_{j_{1}}]\times  \cdots \times[t^{{k}}_{j_{m-1}-1}, t^{{k}}_{j_{n-1}}], \quad\text{with}\quad j_{t}\in\{1,3,\dots, 2(3^{k \,i_{k}})-1\},$$ and $$V^{{k}}_{l}=[s^{k}_{l-1}, s^{k}_{l}]\times[a_{k},b_{k}]^{m-1},\quad\text{    for  }l\in\{1,3,\dots, 2( {3^{k i_{k}(m-1)}})-1\},$$  with $$H^{{k}}_{j_{1},j_{2},\dots,j_{m-1}}\subseteq \phi(V^{{k}}_{l})^{\circ}\quad\text{for some } l.$$ 
We can assume that $E_{k}=\frac{1}{i_{k}^{2}}$ and $H^{{k}}_{j_{1},j_{2},\dots,j_{m-1}}= \phi(V^{{k}}_{l}).$   Set $$U_{l}^{k}=V_{l}^{k}\quad\text{for }l=1,\dots, 2(3^{ki_{k}})-2\quad \text{and}\quad U_{2(3^{k})-1}^{k}=\bigcup_{l=2(3^{k})-1}^{2(3^{ki_{k}({m-1})})-1}V_{l}^{k}.$$
For each $j=1,\dots, m$, let $i_{j}\in\{1,\dots,  2(3^{ki_{k}})-1 \}$   
 and  take $$C^{k}_{i_{1},\dots,i_{m}} =  H^{k}_{i_{1},i_{2},\dots,i_{m-1}}\cap U^{k}_{i_{m}}    .$$  
For $t=1,\dots,n$, let  $i_{1}^{(t)},$ $ \dots,i_{m}^{(t)}\in\{1,3,5,\dots, 2(3^{ki_{k}}) -1\}$  and      
set   \begin{align*} 
 C^{k}_{ i_{1}^{(2)},\dots, i_{m}^{(2)} ,i_{1}^{(1)},\dots, i_{m}^{(1)}} &=\varphi^{-1}\left[\varphi\left(C^{k}_{  i_{1}^{(2)},\dots, i_{m}^{(2)}}\right)\cap C^{k}_{ i_{1}^{(1)},\dots,i_{m}^{(1)}}\right]\\
  &\vdots\\
 C^{k}_{ i_{1}^{(n)},\dots,i_{m}^{(n)},  \dots,    i_{1}^{(1)},\dots, i_{m}^{(1)}} &=\varphi^{-(n-1)}\left[\varphi^{n-1}\left(C^{k}_{ i_{1}^{(n)},\dots,i_{m}^{(n)}, \dots,  i_{1}^{(2)},\dots,i_{m}^{(2)}}\right)\cap C^{k}_{ i_{1}^{(1)},\dots,i_{m}^{(1)}}\right]
 \end{align*} 
   Furthermore,  set $$ \tilde{C}^{k}_{ i_{1}^{(n)},\dots,i_{m}^{(n)},  \dots,    i_{1}^{(1)},\dots, i_{m}^{(1)}} :=\text{exp}\left[ C^{k}_{ i_{1}^{(n)},\dots,i_{m}^{(n)},  \dots,    i_{1}^{(1)},\dots, i_{m}^{(1)}} \right].$$ 
For any $k\geq 1$,   taking   $\varepsilon_{k}=\frac{1}{i_{k}^{2}(2(3^{k \,i_{k} })-1)}$, we have  $$\text{diam}_{n}\left(\tilde{C}^{k}_{  i_{1}^{(n)},\dots,i_{m}^{(n)},  \dots ,i_{1}^{(1)},   \dots,   i_{m}^{(1)}}\right)=\varepsilon_{k},\quad\text{ if  } {i_{t}^{(j)}\in \{{1},3,5,\dots,2({3^{ki_{k}}})-1\}},$$   We have   $3^{ki_{k}nm}$ sets of this form. Hence, 
 \begin{equation*}    \text{H}_{\varepsilon_{k}}^{s}( E_{k}\cap \Omega (\varphi),d_{n},\varphi|_{E_{k}\cap \Omega (\varphi)})\leq \sum_{t=1}^{3^{ki_{k}nm}}\left(\frac{1}{i_{k}^{2}(2(3^{k \,i_{k} })-1)}\right)^{s}=3^{ki_{k}nm}\left(\frac{1}{i_{k}^{2}(2(3^{k \,i_{k} })-1)}\right)^{s}.
 \end{equation*}
 We have \begin{equation*}    \text{H}_{\varepsilon_{k}}^{s}( E_{k}\cap \Omega (\varphi),d_{n},\varphi|_{E_{k}\cap \Omega (\varphi)}) =3^{ki_{k}nm}\left(\frac{1}{i_{k}^{2}(2(3^{k \,i_{k} })-1)}\right)^{s},
 \end{equation*} because $\text{diam}_{n}\left(\tilde{C}_{  i_{1}^{(n)},\dots ,i_{m}^{(n)},  \dots ,i_{1}^{(1)},   \dots ,i_{m}^{(1)}}\right)=\varepsilon_{k}$.    Next,  $$3^{ki_{k}nm}\left(\frac{1}{i_{k}^{2}(2(3^{k \,i_{k} })-1)}\right)^{s}\geq 1\Longleftrightarrow   3^{ki_{k}nm}  \geq   \left({i_{k}^{2}(2(3^{k \,i_{k} })-1)}\right)^{s}\Longleftrightarrow \frac{\log 3^{ki_{k}nm}}{\log \left( {i_{k}^{2}(2(3^{k \,i_{k} })-1)}\right)}\geq s.$$
 Therefore,  \begin{align*}\label{exxample12}   \text{dim}_{\text{H}}(E_{k}\cap\Omega(\varphi),d_{n},\varepsilon_{k})=\frac{\log 3^{ki_{k}nm}}{\log \left( {i_{k}^{2}(2(3^{ki_{k} \,i_{k} })-1)}\right)}\end{align*} and hence \begin{align*}\limsup_{n\rightarrow \infty}\frac{1}{n}\text{dim}_{\text{H}}(N,d_{n},\varepsilon)\geq \lim_{n\rightarrow \infty}\frac{1}{n}\text{dim}_{\text{H}}(E_{k}\cap\Omega(\varphi),d_{n},\varepsilon_{k})=\frac{\log 3^{ki_{k}n}}{\log \left( i_{k}^{2}(2(3^{k \,i_{k} })-1)\right)}.\end{align*} 
 Thus, \begin{equation*} 
   \overline{\text{mdim}}_{\text{H}}(N,d,\varphi)\geq \lim_{k\rightarrow\infty}\frac{\log 3^{ki_{k}m}}{\log \left(i_{k}^{2}(2(3^{k \,i_{k} })-1)\right)}= \lim_{k\rightarrow\infty}\frac{\log 3^{ki_{k}m}}{\log  3^{k \,i_{k} } }=m,\end{equation*}  
 therefore $ \overline{\text{mdim}}_{\text{H}}(N,d,\varphi)=m$, which proves the first part of the theorem. It follows from Lemma \ref{ffeaqw} that $\mathcal{G}_{m}$ contains a residual subset.
\end{proof}

 \end{document}